\def\timestamp{%
Time-stamp: <erdos-homogeneity.tex: Monday 12-02-2018 at 21:31:59 (cet)>}
\def\stripname Time-stamp: <#1 #2>{#2}
\edef\filedate{\expandafter\stripname\timestamp}
\documentclass[a4paper]{amsart}[2015/03/04]
\usepackage{amsrefs}

\newcommand\JFM[1]{ JFM~#1}
\DeclareMathSymbol\N 0{AMSb}{`N}
\DeclareMathSymbol\PP 0{AMSb}{`P}
\DeclareMathSymbol\R 0{AMSb}{`R}
\DeclareMathSymbol\Q 0{AMSb}{`Q}
\let\emptyset \undefined
\let\ge       \undefined
\let\le       \undefined
\DeclareMathSymbol\restr\mathbin{AMSa}{"16}
\DeclareMathSymbol\le   \mathrel{AMSa}{"36}    
\DeclareMathSymbol\ge   \mathrel{AMSa}{"3E}    
\DeclareMathSymbol\emptyset\mathord{AMSb}{"3F}
\newcommand\cee{\mathfrak{c}}
\newcommand\orpr[2]{\langle{#1},{#2}\rangle}
\newcommand\ceeseq[1]{\langle{#1}_\alpha:\alpha<\cee\rangle}
\newcommand\ceeseqq[1]{\langle{#1}:\alpha<\cee\rangle}
\newcommand\bigceeseqq[1]{\bigl<{#1}:\alpha<\cee\bigr>}
\newcommand\norm[1]{\lVert#1\rVert}
\newcommand\bignorm[1]{\bigl\|#1\bigr\|}
\newcommand\abs[1]{\lvert#1\rvert}

\newcommand\preim{^\gets}
\let\epsilon\varepsilon
\newtheorem{theorem}{Theorem}[section]

\newtheorem{lemma}[theorem]{Lemma}

\numberwithin{equation}{section}

\theoremstyle{definition}
\newtheorem*{question}{Question}

\theoremstyle{remark}
\newtheorem{remark}[theorem]{Remark}

\DeclareSymbolFont{cmmib}{OML}{cmm}{b}{it}
\def\newboldletter#1#2{\DeclareMathSymbol#1{0}{cmmib}{`#2}}

\newboldletter\va a
\newboldletter\vb b
\newboldletter\vx x
\newboldletter\vy y
\newboldletter\vX X

\begin{document}

\title[Homogeneity and rigidity in Erd\H{o}s spaces]%
      {Homogeneity and rigidity in Erd\H{o}s spaces}

\author{Klaas Pieter Hart}

\address{Faculty EEMCS\\TU Delft\\
         Postbus 5031\\2600~GA {} Delft\\the Netherlands}
\email{k.p.hart@tudelft.nl}
\urladdr{http://fa.its.tudelft.nl/\~{}hart}

\author{Jan van Mill}
\address{KdV Institute for Mathematics\\
         University of Amsterdam\\
         P.O. Box 94248\\
         1090~GE {} Amsterdam\\
         The Netherlands}
\email{j.vanmill@uva.nl}

\date{\filedate}

\dedicatory{To the memory of Bohuslav Balcar}



\begin{abstract}
We investigate the homogeneity of topological subspaces of separable 
Hilbert space, akin to the spaces with all points rational or all points
irrational, so-called Erd\H{o}s spaces.
We provide a non-homogeneous example, that is based on one set of 
coordinates using, and a rigid example, based on a sequence of coordinate sets.
\end{abstract}

\maketitle

\section*{Introduction}

We let $\ell_2$ denote real separable Hilbert space, that is
$$
\ell_2=\{\vx\in\R^\infty:\sum_{i\in\omega}x_i^2<\infty\}.
$$
In this paper we consider (topological) subspaces of $\ell_2$ that
are defined by taking a sequence $\vX=\langle X_i:i\in\omega\rangle$ of subsets
of~$\R$ and then defining
$$
E(\vX)= \{\vx\in\ell_2:(\forall i)(x_i\in X_i)\}.
$$
If all $X_i$ are equal to one fixed set~$X$ we simply write $E(X)$.
Since $E(\R)$ is just $\ell_2$ itself we henceforth tacitly assume 
that $X\neq\R$ when we deal with a single set~$X$.

These subspaces are generally known as Erd\H{o}s spaces because Erd\H{o}s
showed that $E(S)$ and $E(\Q)$~are a natural examples of totally disconnected
spaces of dimension~one that are also homeomorphic to their own 
squares~\cite{MR0003191}.
Here $S$ denotes the convergent sequence $\{\frac1n:n\in\N\}\cup\{0\}$ 
and $\Q$~denotes the set of rational numbers.
These two spaces have been the object of intense study, Chapter~2 
of~\cite{MR2742005} summarizes much of the earlier history and contains 
references to, among others, a proof that $E(S)$ and~$E(\PP)$ are homeomorphic,
where $\PP$~denotes the set of irrational numbers.

The purpose of this paper is to see what can be said of the spaces~$E(X)$
and $E(\vX)$ in terms of homogeneity and rigidity.
One might think that $E(X)$ is always homogeneous, certainly in the
light of the result of Lawrence from~\cite{MR1458308} that states that 
an infinite power of a zero-dimensional subspace of~$\R$ is always homogeneous.
There are two important differences though: $X^\omega$~is a much larger subset
of~$\R^\omega$ than~$E(X)$, and the topology of~$E(X)$ is finer than
the product topology. 

The standard examples $E(\Q)$ and $E(\PP)$ are homogeneous, they are
even (homeomorphic to) topological groups.
Note that this shows that $E(S)$~is homogeneous, even though $S$~is not 
of course.
 
In the case of a single set one can say for certain that $E(X)$ is not rigid:
any permutation of~$\omega$ induces an autohomeomorphism of~$E(X)$.
These are not the only `easy' autohomeomorphisms of~$E(X)$: 
assume there is a real number~$r$ not in $X$ such that both $X\cap(r,\infty)$
and $X\cap(-\infty,r)$ are nonempty and consider the clopen subset
$C=\{\vx\in E(X): x_0,x_1>r\}$ of~$E(X)$.
One can define $f:E(X)\to E(X)$ to be the identity outside~$C$ and
have $f(\vx)=(x_1,x_0,x_2,\ldots)$ if~$\vx\in C$.
If, as is quite often the case, $X$~has a dense complement in~$\R$ then
one can create many `easy' autohomeomorphisms in this way and we are forced
to conclude that the notion of a `trivial' autohomeomorphism of~$E(X)$ may be 
hard to pin down.

In Section~\ref{sec.norm-pres} we construct a subset~$X$ of~$\R$ such that
$E(X)$~has a rather small set of autohomeomorphisms: all of them
must be norm-preserving.

This of course raises the question whether this can be sharpened to:
every autohomeomorphism of~$E(X)$ must be norm-preserving \emph{and}
all spheres centered at the origin are homogeneous.
We will comment on this after the construction.

To obtain a truly rigid space of the form $E(\vX)$ one must have all 
sets~$X_i$ distinct, for otherwise exchanging two coordinates would result
in a non-trivial autohomeomorphism.
In Section~\ref{sec.rigid} we exhibit a sequence~$\vX$ for which 
$E(\vX)$~is rigid.

The constructions use Sierpi\'nski's method of killing homeomorphisms
from~\cite{zbMATH02550962}, which in turn is based on Lavrentieff's 
theorem from~\cite{Lavrentieff1924}.
The latter theorem states that a homeomorphism between two subsets, 
$A$ and~$B$, of a metrizable space can be extended to a homeomorphism
between $G_\delta$-subsets, $A^*$ and~$B^*$, that contain $A$ and~$B$,
respectively.
It is well-known that a separable metric space, like $\ell_2$, contains
continuum many $G_\delta$-subsets and that each such set admits continuum many
continuous functions into~$\ell_2$.
As will be seen below this will allow us to kill all unwanted homeomorphisms
is a recursive construction of length~$\cee$.

We shall conclude this note with some questions and suggestions for 
further research.

\section{A non-homogeneous Erd\H{o}s space}
\label{sec.norm-pres}

We shall show that there is a subset~$X$ of~$\R$ for which $E(X)$~is 
not homogeneous.
In fact our~$X$ will be such that the autohomeomorphisms of~$E(X)$ must
be norm-preserving.
As observed in the introduction we cannot go all the way and make $E(X)$ rigid:
every permutation of~$\N$ induces a unitary operator on~$\ell_2$ that maps 
$E(X)$ to itself.
Thus, the autohomeomorphism group of~$E(X)$ contains, at least, the
symmetry group~$S_\N$.

\smallskip
We shall construct a dense subset~$X$ of~$\R$ in a recursion of length~$\cee$.
The set~$E(X)$ will then be dense in~$\ell_2$.
If $f:E(X)\to E(X)$ is an autohomeomorphism then we can apply Lavrentieff's
theorem to find a $G_\delta$-set~$A$ that contains~$E(X)$ and an 
autohomeomorphism~$\bar f$ of~$A$ that extends~$f$.
By continuity the map~$\bar f$ is norm-preserving iff $f$~is.
This tells us how we can ensure that $E(X)$~has norm-preserving 
autohomeomorphisms only: make sure that whenever $A$~is a dense $G_\delta$-subset
of~$\ell_2$ that contains~$E(X)$ and $f:A\to A$ is an autohomeomorphism that 
is \emph{not} norm-preserving then $E(X)$~is \emph{not} invariant under~$f$.

To make our construction run a bit smoother we note that it suffices to ensure
that every autohomeomorphism $f:E(X)\to E(X)$ does not increase norms anywhere,
that is, it satisfies $\norm{\vx}\ge\bignorm{f(\vx)}$ for all~$\vx$.
For if $f$~is an autohomeomorphism then so is its inverse~$f^{-1}$ and
from $(\forall\vx)\bigl(\norm{\vx}\ge\bignorm{f^{-1}(\vx)}\bigr)$
we then deduce $(\forall\vx)\bigl(\bignorm{f(\vx)}\ge\norm{\vx}\bigr)$.

\smallskip
We enumerate the set of pairs $\orpr Af$, where $A$~is a dense
$G_\delta$-subset of~$\ell_2$ and $f$~is an autohomeomorphism of~$A$ that 
increases the norm somewhere as~$\bigceeseqq{\orpr{A_\alpha}{f_\alpha}}$.

By transfinite recursion we build increasing sequences~$\ceeseq{X}$ 
and~$\ceeseq{Y}$ subsets of~$\R$ such that for all~$\alpha$
\begin{enumerate}
\item $|X_\alpha\cup Y_\alpha|<\cee$, 
\item $X_\alpha\cap Y_\alpha=\emptyset$, and
\item if $E(X_\alpha)\subseteq A_\alpha$ then there is a point~$\vx_\alpha$ 
      in~$A_\alpha$ such that $X_{\alpha+1}$ consists of~$X_\alpha$ and
      the coordinates of~$\vx_\alpha$, and 
      $Y_{\alpha+1}$~consists of~$Y_\alpha$ and at least one coordinate 
      of~$f_\alpha(\vx_\alpha)$.
\end{enumerate}
To see that this suffices let $X=\bigcup_{\alpha<\cee}X_\alpha$ and 
assume $f$~is an autohomeomorphism of~$E(X)$ that increases the norm
of at least one point.
Apply Lavrentieff's theorem to extend~$f$ to an autohomeomorphism~$\bar f$
of a $G_\delta$-set $A$ that contains~$E(X)$.
Then $A$~is dense and $\bar f$~increases the norm of at least one point
so there is an~$\alpha$ such that $\orpr{A}{\bar f}=\orpr{A_\alpha}{f_\alpha}$.
But now consider the point~$\vx_\alpha$.
It belongs to $E(X_{\alpha+1})$ and hence to~$E(X)$; 
on the other hand one of the coordinates of~$f_\alpha(\vx_\alpha)$ belongs 
to~$Y_{\alpha+1}$ and it follows that $f_\alpha(\vx_\alpha)\notin E(X)$
as $Y_{\alpha+1}\cap X=\emptyset$.
This shows that $\bar f$ does not extend~$f$, 
as $f(\vx_\alpha)$~must be in~$E(X)$, which is a contradiction.

To start the construction let $X_0=\Q$, to ensure density of~$E(X)$,
and $Y_0=\emptyset$.

At limit stages we take unions, so it remains to show what to do at 
successor stages.

To avoid having to carry the index $\alpha$ around all the time we formulate
the successor step as the following lemma, in which $Z$ plays the role
of the union~$X_\alpha\cup Y_\alpha$.

\begin{lemma}\label{lemma.successor.step}
Let $A$ be a dense $G_\delta$-subset of~$\ell_2$ and let $f:A\to A$ be 
an autohomeomorphism that increases the norm of at least one point.
Furthermore let $Z$ be a subset of\/~$\R$ of cardinality less than~$\cee$.
Then there is a point~$\vx$ in~$A$ such that 
\begin{enumerate}
\item none of the coordinates of $\vx$ and $f(\vx)$ are in $Z$, and
\item at least one coordinate of~$f(\vx)$ is not among the coordinates
      of~$\vx$ itself.
\end{enumerate}
\end{lemma}

\begin{proof}
We take $\va\in A$ such that $\norm{\va}<\bignorm{f(\va)}$.

First we show that we can assume, without loss of generality, that $\va$~has
two additional properties:
1)~all coordinates of~$f(\va)$ are non-zero, and
2)~all coordinates of~$f(\va)$ are distinct.

This follows from the fact that the following sets are closed and nowhere 
dense in~$\ell_2$:
\begin{enumerate}
\item $\{\vx\in\ell_2:x_i=0\}$ for every~$i$, and
\item $\{\vx\in\ell_2:x_i\neq x_j\}$ whenever $i<j$.
\end{enumerate}
By continuity and because $A$ is a dense $G_\delta$-subset of~$\ell_2$ we may 
choose $\va$ so that $f(\va)$ is not in any one of these sets.

We claim that there is an $i\in\N$ such that $f(\va)_i\neq a_j$ for all~$j$.  
If not then there is for each~$i$ a (smallest) $k_i$ such 
that $f(\va)_i=a_{k_i}$.
Because all coordinates of~$f(\va)$ are distinct the map $i\mapsto k_i$ must 
be injective.
But then
$$
\sum_{i=0}^\infty f(\va)_i^2 = \sum_{i=0}^\infty a_{k_i}^2\le \sum_{j=0}^\infty a_j^2,
$$
which contradicts our assumption that $\bignorm{f(\va)}>\norm{\va}$.

Fix an $i$ as above.
Since $\lim_ia_i=0$ and $f(\va)_i\neq0$ we can take $\varepsilon>0$ such that
$\abs{f(\va)_i-a_j}\ge3\varepsilon$ for all~$j$.

By continuity we can take $\delta>0$ such that $\delta\le\varepsilon$ 
and such that 
$\norm{\vx-\va}<\delta$ implies $\bignorm{f(\vx)-f(\va)}<\varepsilon$.

By the triangle inequality we find that when $\norm{\vx-\va}<\delta$
we have $\abs{f(\vx)_i-x_j}\ge\varepsilon$ for all~$j$.

Now we apply Lemma~4.2 from~\cite{MR742164}.
The conditions of this lemma are that we have a separable completely metrizable
space, for this we take $M=B(\va,\delta)\cap A$.
Next we need a family of countably many continuous functions to one space, 
for this we take the coordinate maps $\pi_j:\vx\mapsto x_j$ and their
compositions with~$f$, that is $\rho_j:\vx\mapsto f(\vx)_j$; the codomain
is the real line~$\R$.
Finally, we need to know that whenever $C\subseteq\R$ is countable the
complement of
$$
\bigcup_{j\in\omega}\bigl(\pi_j\preim[C]\cup\rho_j\preim[C]\bigr)
$$
in~$M$ is not countable.
This is true because the preimages of points under the~$\pi_j$ and the~$\rho_j$
are nowhere dense, so that the complement is a dense $G_\delta$-subset of~$M$.

The conclusion then is that there is a (copy of the) Cantor set~$K$ 
inside $B(\va,\delta)\cap A$ such that all maps~$\pi_j$ and~$\rho_j$
are injective.

Because $|K|=\cee$ this then yields many points $\vx\in K$ such that
$x_j,f(\vx)_j\notin Z$ for all~$j$.
All these points are as required.
\end{proof}

\begin{remark}\label{rem.spheres}
One would like to make this example as sharp as possible, for example by 
making all spheres centered at the origin homogeneous.
This seems harder than one might think at first.
Some straightforward modifications of the construction in this section 
will go agley.

One might try to add some unitary operators as autohomeomorphisms of~$E(X)$.
To keep the sets $X_\alpha$ and $Y_\alpha$ small these should introduce as 
few new coordinates as possible.
But even a simple transformation of the first two coordinates as given by
$u_0=\frac35x_0+\frac45x_1$ and $u_1=-\frac45x_0+\frac35x_1$ is potentially
quite dangerous.
For if $x_0=x_1$ then $u_0=\frac95x_0$ and $u_1=-\frac15x_0$.
This shows that if $E(X)$ is to be invariant under just this operator the
set $X$ itself must be invariant under scaling by~$\frac95$ and~$-\frac15$.
This would introduce norm-changing autohomeomorphisms of~$E(X)$.

Another possibility would be to make $E(X)$ invariant under 
standard reflections:
if $\norm\va=\norm\vb$ then
$$
R(\vx)=\vx-2\frac{\vx\cdot(\va-\vb)}{(\va-\vb)\cdot(\va-\vb)}(\va-\vb)
$$
defines the reflection in the perpendicular bisector of~$\va$ and~$\vb$.
Unfortunately this would mean that as soon as $X_\alpha$ is infinite
and dense there would be $\cee$~many of these maps and hence $\cee$~many
coordinates to avoid.
This would make it quite difficult to keep the sets $X$ and~$Y$ above disjoint.
\end{remark}

\begin{remark}
It is relatively easy to create situations where some spheres are not 
homogeneous.
Simply take a set $X$ that has $0$ as an element \emph{and} as an accumulation
point, and an isolated point~$x$, e.g., the convergent sequence~$S$ mentioned 
in the introduction.
In~$E(X)$ the sphere $H=\{\vy\in E(X):\norm\vy=\abs{x}\}$ is not homogeneous.

Indeed, the point $\vx=\langle x,0,0,\ldots\rangle$ is isolated in~$H$.
To see this take $\epsilon>0$ such that $\{x\}=X\cap(x-\epsilon,x+\epsilon)$
and consider any $\vy\in H\setminus\{\vx\}$.
Then $y_0\neq x$ because $\sum_iy_i^2=x^2$, hence $\abs{y_0-x}\ge\epsilon$
and also $\norm{\vy-\vx}\ge\epsilon$. 

On the other hand, using a non-trivial convergent sequence in~$X$ with
limit~$0$ it is an elementary exercise to construct a non-trivial
convergent sequence in~$H$.
\end{remark}

\section{A rigid example}
\label{sec.rigid}

In this section we construct a rigid Erd\H{o}s space.
As noted before, in this case we need a sequence 
$\vX=\langle X_i:i\in\omega\rangle$
of subsets of~$\R$ simply because we need to disallow permutations of 
coordinates as autohomeomorphisms.

The construction is similar to, but easier than, that in 
Section~\ref{sec.norm-pres}.
We list the set of pairs $\orpr Af$, where $A$ is a dense $G_\delta$-subset
of~$\ell_2$ and $f:A\to A$ is a homeomorphism that is not the identity,
as~$\bigceeseqq{\orpr{A_\alpha}{f_\alpha}}$.

We now build countably many increasing sequences $\ceeseqq{X_{i,\alpha}}$
of subsets of~$\R$, one for each~$i$ and one countably many auxiliary 
sequences~$\ceeseqq{Y_{i,\alpha}}$ such that
$Y_{i,\alpha}\cap X_{i,\alpha}=\emptyset$ for all~$i$ and all~$\alpha$.

We start with a sequence $\langle X_{i,0}:i\in\omega\rangle$
of pairwise disjoint countable dense subsets of~$\R$ 
and $Y_{i,0}=\emptyset$ for all~$i$.

At limit stages we take unions and at a successor stage we consider
$\vX_\alpha=\langle X_{i,\alpha}:i\in\omega\rangle$, the
corresponding Erd\H{o}s space $E(\vX_\alpha)$, 
and the pair~$\orpr{A_\alpha}{f_\alpha}$. 
In case $E(\vX_\alpha)\subseteq A_\alpha$ we take a point
$\va\in A_\alpha$ such that $f_\alpha(\va)\neq\va$ and we fix a coordinate~$j$ 
such that $a_j\neq f_\alpha(\va)_j$ and let $\epsilon=\abs{a_j-f(\va)_j}/2$.
Because $A_\alpha$ is a dense $G_\delta$-set there are many points in~$A_\alpha$
within distance~$\epsilon$ of~$f_\alpha(\va)$ whose $j$-th coordinate is not
in~$X_{j,\alpha}$; by continuity of the bijection~$f_\alpha$ we can
assume that $f_\alpha(\va)_j\notin X_{j,\alpha}$.

We then put $X_{i,\alpha+1}=X_{i,\alpha}\cup\{a_i\}$ for all~$i$,
and $Y_{i,\alpha+1}=Y_{i,\alpha}$ for all~$i\neq j$, 
and $Y_{j,\alpha+1}=Y_{j,\alpha}\cup\{f(\va)_j\}$.

In the end we let $X_i=\bigcup_{\alpha<\cee}X_{i,\alpha}$.

As in the previous section if $f$ is an autohomeomorphism of~$E(\vX)$
that is not the identity then there is an $\alpha$ such that 
$E(\vX)\subseteq A_\alpha$ and $f_\alpha$~extends~$f$.
However, for the point~$\va$ chosen at that stage we have $\va\in E(\vX)$
and $f_\alpha(\va)\notin E(\vX)$.

\section{Some questions}

In this last section we formulate two questions that we deem of particular
interest in the context of homogeneity and rigidity in Erd\H{o}s spaces.

\begin{question}
Given a subset $X$ of $\R$ with a dense complement, investigate
what the `trivial' autohomeomorphisms of~$E(X)$ should be.
\end{question}

This is the kind of question that one asks of any kind of structure:
what are the automorphisms?
Since at least one $E(X)$ has norm-preserving autohomeomorphisms only
we know that `trivial' should imply that property.

\begin{question}
Is there a set $X$ such that $E(X)$ has norm-preserving autohomeomorphisms
only \emph{and} such that all spheres centered at the origin are homogeneous?  
\end{question}

Note that one can split the last condition into two possibilities:
one can ask whether the spheres can be made homogeneous as spaces in their own
right or whether one can use autohomeomorphisms of~$E(X)$ to establish
their homogeneity.
The failed attempts described in Remark~\ref{rem.spheres} were of 
the latter kind.

\end{document}